\documentclass{article}
\usepackage[utf8]{inputenc}
\usepackage[T1]{fontenc}
\usepackage{bm, amsmath, amssymb, amsthm} 
\usepackage{color}

\newtheorem{theorem}{Theorem}
\newtheorem{proposition}[theorem]{Proposition}
\newtheorem{lemma}[theorem]{Lemma}

\newcommand{\R}{\mathbb{R}}
\newcommand{\Q}{\mathbb{Q}}

\def\<{\langle}
\def\>{\rangle}
\def\d{\partial}
\def\be{\begin{equation} }
\def\ee{\end{equation} }

\title{A class of minimal hypersurfaces in $\Q^N(c)$}
\author{C. do Rei Filho and S. Canevari}

\begin{document}

\maketitle
\begin{abstract}
We prove that a minimal hypersurfaces $f\colon M^{3} \to \Q^4(c)$ with nonzero three distinct principal curvature cannot be isometrically immersed in $\Q^4(\tilde{c}), \  \tilde{c}\neq c$. In the other cases, we present a description of such hypersurfaces. \\

\end{abstract}

\section{Introduction}

We denote by $\Q^N(c)$ a Riemannian space form of dimension $N$ and constant sectional curvature $c$. By a hypersurface $f\colon M^{n} \to \Q^{n+1}(c)$ we always mean  an isometric immersion of a \emph{Riemannian} manifold
$M^n$ of dimension $n$ into $\Q^{n+1}(c)$. 
We say that $f$ is a minimal hypersurface whose mean curvature is zero at all points in $M^n$.

One of the main purposes of  this paper is to address the following problem \\ \vspace{1ex}

\emph{For which minimal hypersurfaces $f\colon M^{n} \to \Q^{n+1}(c)$ of dimension $n\geq 2$ does there exist another isometric immersion  $\tilde{f}\colon M^{n} \to \Q^{n+1}(\tilde{c})$
with $\tilde{c}\neq c$?}\\  \vspace{1ex}

For $n=2$, Lawson \cite{La} show that given $f$ as in problem, exists a 2-parameter family the $\tilde{f}$ whose mean curvature is given by $\tilde{H}=\sqrt{c-\tilde{c}}.$ For $n \geq 4$, the problem is also solved. In fact, such a problem was studied by do Carmo and Dajczer in \cite{dcd2}, without the hypothesis that one of the immersions were minimal, and it was shown that $M^n$ is conformally flat. Already in \cite{dcd}, the same authors showed that the minimal rotation hypersurface is the only 
 minimal isometric immersions of conformally flat manifolds. In this article, we consider the problem for hypersurfaces of dimension $n = 3$.

Note that trivial solution the problem are inclusions $i\colon \Q^3(c) \to \Q^4(c)$ and $\tilde{i}\colon \Q^3(c) \to \Q^4(\tilde{c})$, $c > \tilde{c}$, umbilical non totally geodesic. 

To build more interesting solutions to the problem, we call by a \emph{$r\delta-$catenary} in $\Q^2(k)\subset \R^3_{\epsilon_0}$ with respect to a unit vector $v \in \R^3_{\epsilon_0}$ is a unit-speed curve $\gamma\colon I\to \Q^2(k)\subset \R^3_{\epsilon_0}$ such that the height function $\gamma_v=\<\gamma, v\>$ satisfies $$\gamma_v\gamma_v''+3r\gamma_v+2\gamma_v'^2-2\delta=0,$$ where $\delta \in \{-1,0,1\}$ and $r \in \R$. Here $\epsilon_0=0$ or $1$, corresponding to $k>0$ or $k<0$, respectively.


A solution the problem more interesting is built as follows: Let $f\colon M^3 \to \Q^4(c)$ be a minimal rotation hypersurface. We have (see, \cite[Corollary 4.4]{dcd}) that profile curve $\beta\colon I \to \Q^2(c)\subset \R^3_{\epsilon_0}$ of $f$ is a $c\delta-$catenary with respect to a unit vector $e_1 \in \R^3_{\epsilon_0}$. Let $\beta_1$ a height function of $\beta$ with respect to a unit vector $e_1 \in \R^3_{\epsilon_0}$. Now consider the rotation hypersurface $\tilde{f}\colon M^3 \to \Q^4(\tilde{c})$ built as in \cite{dcd}, through the function $\beta_1$. Follows that the first form of $M^3$ induced by $f$ and $\tilde{f}$ coincide (see, \cite[Corollary 3.8]{dcd}).

The following result provides the solutions the problem.

\begin{theorem}\label{main1}
Let $f:M^3 \to \Q^4(c)$ be a minimal hypersurface for which there exists an isometric immersion $\tilde{f}:M^3 \to \Q^4(\tilde{c})$ with $c \neq \tilde{c}$. Then either $f$ has zero as one of its principal curvatures or it has a principal curvature with multiplicity at least two. Moreover,
\begin{enumerate}
    \item [a)] If $f$ has a principal curvature with multiplicity three then $c>\tilde{c}$, $f$ is totally geodesic and $\tilde f$ is umbilic non totally geodesic.
    \item [b)] If $f$ has a principal curvature with multiplicity two then $f$ is a rotation hypersurface whose profile curve is a $c\delta-$catenary in a totally geodesic surface $\Q^2(c)$ of $\Q^4(c)$ and $\tilde{f}$ is a rotation hypersurface whose profile curve is a $c\tilde{\delta}-$catenary in a totally geodesic surface $\Q^2(\tilde{c})$ of $\Q^4(\tilde{c})$.
    \item [c)] If $f$ has a principal curvature zero then $c > \tilde{c}$, $f$ is a generalized cone over a minimal surface in an umbilical hypersurface $\Q^3(\bar{c})$ of $\Q^4(c)$, $\bar{c} \geq c$, and $\tilde{f}$ is a rotation hypersurface whose profile curve is a $c-$helix in a totally geodesic surface $\Q^2(\tilde{c})$ of $\Q^4(\tilde{c})$.
\end{enumerate}
\end{theorem}

 By a \emph{generalized cone} over a surface 
$g\colon M^2\to \Q^3(\bar c)$ in an umbilical hypersurface $\Q^3(\bar c)$ of $\Q^{4}({c})$, $\bar c\geq c$, we mean the 
hypersurface parametrized by (the restriction to the subset of regular points of) the map $G\colon M^2\times \R\to \Q^{4}({c})$ given by $$G(x, t)=\exp_{g(x)}(t\xi(g(x))),$$ where $\xi$ is a unit normal vector field to the inclusion 
$i\colon \Q^3(\bar c)\to \Q^{4}({c})$ and $\exp$ is the exponential map of  $\Q^{4}({c})$. We say that $\gamma$ is a \emph{$c$-helix} on $\Q^2(k)$ with respect to a unit vector $v$ if $\gamma_v$ satisfies $\gamma_v''+c\gamma_v=0$. 

This article is organized as follows: In section 2, we present some results that will be useful in the other sections. In section 3, we study problem for the case of three distinct and non-zero principal curvatures and show that, in this case, there is no solution to the problem. In Section 4, we study the problem for the principal curvature case with multiplicities or zero, describe the solutions of the problem and prove the Theorem \ref{main1}. At the end of the article, there is an appendix with some extensive expressions that we have avoided leaving in the course of the previous sections.
 
\section{Preliminaries}

First we  recall the notion of holonomic hypersurfaces.  One says that a  hypersurface $f\colon M^{n} \to \Q^{n+1}(c)$  is   \emph{holonomic} if $M^n$ carries  global orthogonal coordinates  $(u_1,\ldots, u_{n})$  such that the coordinate vector fields 
$\d_j=\dfrac{\partial}{\partial u_j}$ diagonalize the second fundamental form $I\!I$ 
 of $f$. 

Set $v_j=\|\d_j\|$
 and define $V_{j} \in C^{\infty}(M)$ by 
 $II(\d_j, \d_j)=V_jv_j$, $1\leq j\leq n$. 
Then the first and second fundamental forms of $f$ are
\be\label{fundforms}
I=\sum_{i=1}^nv_i^2du_i^2\,\,\,\,\mbox{and}\,\,\,\,\,I\!I=\sum_{i=1}^nV_iv_idu_i^2.
\ee
Denote $v=(v_1,\ldots, v_n)$ and  $V=(V_{1},\ldots, V_n)$. We call  $(v,V)$ the pair associated to $f$. 
The next result is well known.

\begin{proposition}\label{prop.hol}
 The triple $(v,h,V)$, where 
 $\displaystyle h_{ij}=\frac{1}{v_i}\frac{\partial v_j}{\partial u_i},$
 satisfies the  system of PDE's
 \begin{equation}\label{sist.hol}
\left\{\begin{array}{l}
\displaystyle ({\text{i}}) \  \frac{\partial v_i}{\partial u_j}= h_{ji}v_j, \quad  ({\text{ii}}) \ \frac{\partial h_{ij}}{\partial u_i} + \frac{\partial h_{ji}}{\partial u_j} +h_{ki}h_{kj}+ V_iV_j + c \, v_iv_j=0, \vspace{.18cm}\\

\displaystyle ({\text{iii}}) \ \frac{\partial h_{ik}}{\partial u_j}=h_{ij}h_{jk}, \quad  ({\text{iv}}) \ \frac{\partial V_i}{\partial u_j}=h_{ji}V_j, \hspace{.5cm} 1\leq i\neq j\neq k\neq i \leq n.
\end{array}\right.
\end{equation}
Conversely, if $(v,h,V)$ is a solution of $(\ref{sist.hol})$ on a simply connected open subset   $U \subset \R^{n}$, with  $v_i \neq 0$ everywhere  for all $1\leq i\leq n$, then there exists a holonomic hypersurface $f\colon U \to  \R^{n+1}$  whose first and second fundamental forms are given by  $(\ref{fundforms}).$
\end{proposition}

The following result characterize of hypersurfaces $f\colon M^3 \to \Q^4(c)$ with three distinct
principal curvatures does there exist another isometric immersion $\tilde{f} \colon M^3 \to \Q^4(\tilde{c})$ with $c \neq \tilde{c}$. 

\begin{theorem}[\cite{ct}] \label{main.theo} Let $f\colon\,M^3\to\Q^{4}(c)$ be a simply connected  holonomic hypersurface whose associated pair $(v, V)$ satisfies 
\be\label{main.theo.eq}
\sum_{i=1}^3\delta_iv_i^2=1, \,\,\,\,\,\,\sum_{i=1}^3\delta_iv_iV_i=0\,\,\,\,\,\mbox{and}\,\,\,\,\,\sum_{i=1}^3\delta_iV_i^2=(c-\tilde{c}), 
\ee
where $\tilde{c}\neq c$ and  $(\delta_1, \delta_2, \delta_3)= (1,-1, 1).$ Then $M^3$ admits an isometric immersion into $\mathbb{Q}^4(\tilde{c}),$ which is unique up to congruence.

Conversely, if $f\colon\,M^3\to\Q^{4}(c)$ is a hypersurface with three distinct principal curvatures
for which there exists an isometric immersion $\tilde{f}\colon\,M^3\to\Q^{4}(\tilde{c}),$ then f is locally
a holonomic hypersurface whose associated pair $(v, V )$ satisfies $(\ref{main.theo.eq}).$
\end{theorem}


\section{Case nonzero distinct principal curvatures}

It will be convenient to use the following Proposition.

\begin{lemma}\label{lem.tecn.}
	Let $f\colon\,M^3\to\Q^{4}(c)$ be a simply connected holonomic hypersurface with three nonzero distinct principal curvature whose associated pair $(v, V)$ satisfies 
	\begin{equation}\label{caract.minima}
		\begin{array}{l}\displaystyle
			\sum_{i=1}^3\delta_iv_i^2=1, \ \sum_{i=1}^3\delta_iv_iV_i=0, \ \sum_{i=1}^3\delta_iV_i^2=(c-\tilde{c}) \ \mbox{and} \ \sum_{i=1}^3{\frac{V_i}{v_i}}=0
		\end{array}
	\end{equation}
	where $c \neq \tilde{c}$ and $(\delta_1,\delta_2,\delta_3)=(1,-1,1).$ Then $v_1 \neq v_3$, 
		\begin{equation}\label{V2V3}
			V_2=\dfrac{v_2(v_1^2-v_3^2)}{v_1(v_2^2+v_3^2)}V_1,\, V_3=-\dfrac{v_3(v_1^2+v_2^2)}{v_1(v_2^2+v_3^2)}V_1
		\end{equation}
		and
		\begin{equation}\label{good.term}
			\Theta:= -v_1^2 + v_1^4 + v_2^2 - 10 v_1^2 v_2^2 + 9 v_1^4 v_2^2 + v_2^4 - 9 v_1^2 v_2^4 \neq 0.
		\end{equation} 	
Moreover:
	\begin{enumerate}	
		\item[a)] If $c-\tilde{c}<0$ then $v_k\neq v_2, \, k  \in \{1,3\},$ $v_i^2-\delta_i\neq 0$ and $3v_i^2-\delta_i\neq 0.$
	
		\item[b)] If $c-\tilde{c}>0$, then $3v_k^2 - \delta_k \neq 0$ and $v_k\neq v_2, \, k  \in \{1,3\},$ in an open and dense subset of $M^3.$
		\end{enumerate}
\end{lemma}
\begin{proof}
 The second and fourth equation in (\ref{caract.minima}) implies
	\begin{equation}\label{ViVj}
			\displaystyle v_i(\delta_j v_j^2-\delta_k v_k^2)V_j+v_j(\delta_i v_i^2-\delta_k v_k^2)V_i=0.
	\end{equation}
Whence follows the equations in (\ref{V2V3}). The fact that $v_1 \neq v_3$ follows from the first equation of (\ref{V2V3}). Moreover to obtain (\ref{good.term}) just substitute the expressions of \eqref{V2V3} in the third equation of (\ref{caract.minima}). 

Suppose initially that $c<\tilde{c}$. If $v_1=v_2$ then $v_3=1$ and the equations in (\ref{V2V3}) are reduced to $$V_2=\frac{v_1^2-1}{v_1^2+1}V_1\,\,\, \text{and}\,\,\, V_3=-\frac{2v_1}{v_1^2+1}V_1.$$ Substituting these expressions into the third equation of (\ref{caract.minima}) we get a contradiction. It is shown in the same way as $v_2\neq v_3$ and as $3v_i^2-\delta_i\neq 0.$ The penultimate information in item $(a)$ follows from the fact that $v_i \neq v_j$, for $i \neq j$.

Now suppose that $c>\tilde{c}$. If $3v_k^2 - \delta_k = 0$ and $k \in \{1,3\},$ follow from (\ref{V2V3}) that $$\dfrac{V_i}{v_i}=\dfrac{V_j}{v_j}=\sqrt{\dfrac{c-\tilde{c}}{2}},\, k\neq i \neq j \neq k,$$ which is a contradiction.

Finally remark that the  triple $(v,h,V),$ where  $\displaystyle h_{ij}=\frac{1}{v_i}\frac{\partial v_j}{\partial u_i},$ satisfies the system of PDE's 

	\begin{equation}\label{pde.sist.hol.min}
		\hspace{-1cm}	\left\{\begin{array}{l}
				\displaystyle ({\text{i}}) \  \frac{\partial v_i}{\partial u_j}= h_{ji}v_j, \vspace{.18cm}\\ ({\text{ii}}) \ \frac{\partial h_{ij}}{\partial u_i} + \frac{\partial h_{ji}}{\partial u_j} +h_{ki}h_{kj}+ V_iV_j+ cv_iv_j=0, \vspace{.18cm}\\
				
				\displaystyle ({\text{iii}}) \ \frac{\partial h_{ik}}{\partial u_j}=h_{ij}h_{jk}, \vspace{.18cm}\\ ({\text{iv}}) \ \frac{\partial V_i}{\partial u_j}=h_{ji}V_j, \hspace{.5cm} 1\leq i\neq j\neq k\neq i \leq 3,\vspace{.18cm}\\
				
				\displaystyle ({\text{v}}) \  \delta_i\frac{\partial v_i}{\partial u_i}+ \delta_jh_{ij}v_j +\delta_kh_{ik}v_k=0, \vspace{.18cm}\\ ({\text{vi}}) \ \delta_i\frac{\partial V_i}{\partial u_i}+ \delta_jh_{ij}V_j +\delta_kh_{ik}V_k=0,
			\end{array}\right.
	\end{equation}
	where equations (i), (ii), (iii) and (iv) follow from $f$ being a holonomic hypersurface with  $(v,V)$ as associated pair (see Proposition \ref{prop.hol}),  while (v) and (vi) follow by differentiating the first and third equation in  (\ref{caract.minima}).

Suppose there is a open $U \subset M^3$ such that $v_1=v_2.$ Then $v_3=1$ and, using (\ref{V2V3}), we can write $\displaystyle V_1=-\frac{1+v_1^2}{2v_1}V_3$ and $\displaystyle V_2=\frac{1-v_1^2}{2v_1}V_3.$ Substituting these data into the third equation of (\ref{caract.minima}), we concluded that  $\displaystyle V_3=\pm \sqrt{\frac{c-\tilde{c}}{2}}.$ Hence, it follows from equations (i), (v) and (vi) in (\ref{pde.sist.hol.min}), that  $h_{13}=h_{23}=0$ and $h_{31}=h_{32}=0$. Substituting this information into equation (ii), we get a contradiction. Likewise, it is shown that $v_3 \neq v_2$ in an open and dense subset $M^3.$
\end{proof}

\begin{proposition}\label{prop.key}
Let $f\colon\,M^3\to\Q^{4}(c)$ be a simply connected holonomic hypersurface with three nonzero distinct principal curvature whose associated pair $(v, V)$ satisfies (\ref{caract.minima}). Let $\alpha=(\alpha_1,\alpha_2,\alpha_3)$ such that
\begin{equation}\label{def.alpha}
\alpha_i=\dfrac{1}{v_j\phi_j}h_{ij}
\end{equation}
where $\displaystyle \phi_j= (\delta_j - v_j^2) (\delta_j - 3 v_j^2),$ with $1 \leq i \neq j \leq 3$. Then the pair $(v,\alpha)$ and $V_1$ satisfy the system of PDE's
\begin{equation}\label{red.system}
\left\{\begin{array}{llll}
\displaystyle \frac{\partial v_i}{\partial u_j} =  v_j v_i\phi_i \alpha_j, \\ \\
\displaystyle \frac{\partial v_i}{\partial u_i} = -\delta_i(\delta_j v_j^2\phi_j + \delta_k v_k^2 \phi_k)\alpha_i, \\ \\ \displaystyle\frac{\partial \alpha_i}{\partial u_j} = v_j(\phi_j-5\phi_k-8\delta_kv_k^2+4)\alpha_i\alpha_j, \\ \\

\displaystyle\frac{\partial \alpha_i}{\partial u_i} =\displaystyle\frac{v_i \phi_i}{2 \phi_j} (-8  \delta_i v_i^2 + 8 \delta_k v_k^2 -  5 \phi_i - \phi_j + 5 \phi_k) \alpha_j^2 \\ \\ + \frac{v_i \phi_i}{2 \phi_k} (-8 \delta_i v_i^2 + 8 \delta_j v_j^2 - 5 \phi_i + 5 \phi_j - \phi_k) \alpha_k^2 \\ \\ - \displaystyle \frac{v_i}{2} (-8 \delta_i v_i^2 - \phi_i + 5 \phi_j + 5 \phi_k) \alpha_i^2 + \displaystyle\frac{c \, v_i (\phi_i - \phi_j - \phi_k)}{2 \phi_j \phi_k} \\ \\ 
 - \displaystyle\frac{1}{2 v_j v_k \phi_j \phi_k} (-v_i\phi_i V_j V_k  + v_j\phi_j V_i  V_k  + v_k \phi_k V_i V_j  ),  \\ \\

\displaystyle \frac{\partial V_1}{\partial u_1} = \displaystyle-\frac{V_1}{v_1}(2v_1^4 - 3v_1^6 - v_3^2 + v_1^2 v_3^2 + v_3^4) \alpha_1,\\ \\
\displaystyle \frac{\partial V_1}{\partial u_i} = \displaystyle -\frac{\delta_i v_i\phi_1}{v_i^2+v_j^2}(v_1^2-\delta_j v_j^2)  V_1\alpha_i, \   i,j \neq 1. 
\end{array}\right.
\end{equation}
where  $i,j,k\in \{1,2,3\}$ with $ i\neq j\neq k\neq i$. As well as the algebraic equations
\begin{equation}\label{compat.eq}
\begin{array}{llll}
\alpha_1 F_1 = 0,  \qquad \alpha_2 F_2 = 0  \qquad \mbox{and} \qquad \alpha_3 F_3 = 0,
\end{array}
\end{equation}       
where the functions  $F_1,F_2,F_3\colon M^3\to \R$ are defined in the appendix.
\end{proposition}
\begin{proof}
By Lema \ref{lem.tecn.}, $\phi_i\neq 0,$ for all $i$. The  triple $(v,h,V),$ where  $\displaystyle h_{ij}=\frac{1}{v_i}\frac{\partial v_j}{\partial u_i},$ satisfies the system of PDE's \eqref{pde.sist.hol.min}. By differentiating the fourth equations in (\ref{caract.minima}) and using the equations of system (\ref{pde.sist.hol.min}) we obtain 
\begin{equation}\label{rel.h-ki-kj}
 v_j\phi_j h_{ik}=v_k\phi_k h_{ij}, \quad 1\leq i\neq j\neq k\neq i\leq 3.
\end{equation}
From  equations (i) and (v) of system (\ref{pde.sist.hol.min}), together with (\ref{def.alpha}), we obtain the derivatives $\displaystyle{\dfrac{\partial v_i}{\partial u_j}},$ for all $i,j$, of system (\ref{red.system}). In a similar way we can use equations (i) and (iii) of system (\ref{pde.sist.hol.min}), the equations already obtained of system (\ref{red.system}) and  (\ref{def.alpha}) to obtain the expressions for the derivatives $\displaystyle{\dfrac{\partial \alpha_i}{\partial u_j}},$ for all $i \neq j$. To compute  $\displaystyle{\dfrac{\partial \alpha_i}{\partial u_i}},$ for all $i$, notice that equation $(ii)$ of system (\ref{pde.sist.hol.min}), together with (\ref{def.alpha}) and the remaining equations of (\ref{red.system}), determine the system of linear equations
\begin{equation}\label{obt. alpha-ii}
\left\{\begin{array}{l}
\displaystyle v_2\phi_2 \frac{\partial \alpha_1}{\partial u_1} + v_1\phi_1 \frac{\partial \alpha_2}{\partial u_2} = b_1 \vspace{.2cm}\\
\displaystyle v_3\phi_3 \frac{\partial \alpha_1}{\partial u_1} + v_1\phi_1 \frac{\partial \alpha_3}{\partial u_3} = b_2 \vspace{.2cm}\\
\displaystyle v_3\phi_3 \frac{\partial \alpha_2}{\partial u_2} + v_2\phi_2 \frac{\partial \alpha_3}{\partial u_3} = b_3 
\end{array}\right.
\end{equation} 
in the unknowns $\displaystyle{\dfrac{\partial \alpha_i}{\partial u_i}},$ for all $i$, and the functions $b_i,$ being given by
\begin{equation*}
\begin{array}{lll}
b_1 &=& v_1 v_2 \phi_2 (5 \phi_2 - 8 v_2^2 - 4) \alpha_1^2 + v_1 v_2 \phi_1 (5 \phi_1 + 8 v_1^2 - 4) \alpha_2^2 \vspace{.2cm}\\ & & + \ v_1 v_2 \phi_1 \phi_2 \alpha_3^2 + V_1 V_2 \\ \\
b_2 &=& v_1 v_3 \phi_3 (5 \phi_3 + 8 v_3^2 - 4) \alpha_1^2 +  v_1 v_3 \phi_1 \phi_3 \alpha_2^2 \vspace{.2cm}\\ & & + \ v_1 v_3 \phi_1 (5 \phi_1 + 8 v_1^2 - 4) \alpha_3^2 + V_1 V_3 \\ \\
b_3 &=& v_2 v_3 \phi_2 \phi_3 \alpha_1^2 + v_2 v_3 \phi_3 (5 \phi_3 + 8 v_3^2 - 4) \alpha_2^2 \vspace{.2cm} \\ & & + \ v_2 v_3 \phi_2 (5 \phi_2 - 8 v_2^2 - 4) \alpha_3^2 + V_2 V_3
\end{array}
\end{equation*}
Then, one can check that  
\begin{displaymath}
\det\left( \begin{array}{ccc}
v_2\phi_2 & v_1\phi_1 & 0 \\
v_3\phi_3 & 0 & v_1\phi_1 \\
0 & v_3\phi_3 & v_2\phi_2
\end{array}\right) = -2v_1v_2v_3\phi_1\phi_2\phi_3 \neq 0, 
\end{displaymath}
and that the fourth equation of (\ref{red.system}) is the unique solution of system (\ref{obt. alpha-ii}).

It remains to show that the pair $(v,\alpha)$  satisfies (\ref{compat.eq}). Computing the mixed derivatives  $\displaystyle{\dfrac{\partial^2 \alpha_i}{\partial u_j\partial u_k} = \dfrac{\partial^2 \alpha_i}{\partial u_k\partial u_j}},$ for all $i,k,j$, from the equations in (\ref{red.system}), we obtain 

\begin{equation*}
\begin{array}{l}
 \displaystyle 0=\frac{\partial^2 \alpha_1}{\partial u_2\partial u_1}-\frac{\partial^2 \alpha_1}{\partial u_1\partial u_2} = 
\frac{2 v_2}{v_1 (v_2^2 + v_3^2)^2 \phi_2^2 \phi_3}\alpha_2 F_2;  
\vspace{0.2cm} \\

\displaystyle 0= \frac{\partial^2 \alpha_1}{\partial u_3\partial u_1}-\frac{\partial^2 \alpha_1}{\partial u_1\partial u_3} = 
\frac{2 v_3}{v_1 (v_2^2 + v_3^2)^2 \phi_2 \phi_3^2}\alpha_3 F_3 ;  
 \vspace{0.2cm} \\

\displaystyle 0=\frac{\partial^2 \alpha_2}{\partial u_2\partial u_1} - \frac{\partial^2 \alpha_2}{\partial u_1\partial u_2} = 
\frac{-2 v_2}{v_1 (v_2^2 + v_3^2) \phi_1^2 \phi_3}\alpha_1 F_1 .   
\end{array}
\end{equation*}
\end{proof}


\begin{proposition}\label{propried.of.solu}
With the  assumptions of Proposition  \ref{prop.key} we have:
\begin{enumerate}
\item[(a)] There does not exist any open subset  $U \subset M^3$ where the functions  $\alpha_1$, $\alpha_2$ and $\alpha_3$ all vanish identically;
\item[(b)] The function $\alpha_1\alpha_2\alpha_3$ vanishes identically on $M^3$;
\item[(c)] There does not exist any open subset  $U \subset M^3$ 
where two of the functions $\alpha_1$, $\alpha_2$ and $\alpha_3$  vanish identically and the other is nowhere vanishing;
\item[(d)] There does not exist any open subset  $U \subset M^3$ where one of the functions $\alpha_1$, $\alpha_2$ and $\alpha_3$  vanishes identically and the others are nowhere vanishing;
\end{enumerate} 
\end{proposition}


\begin{proof}
$(a)$ Let $U \subset M^3$ be an open subset where $\alpha_1$, $\alpha_2$ and $\alpha_3$ all vanish identically. Then the fourth equation in (\ref{red.system}) reduces to 
\begin{equation}\label{a1a2a3zero}
\left\{\begin{array}{lll}
c \, v_1 v_2 v_3 (\phi_1 - \phi_2 - \phi_3) = -v_1 V_2 V_3 \phi_1 + V_1 v_2 V_3 \phi_2 + V_1 V_2 v_3 \phi_3, \vspace{.2cm} \\
 c \, v_1 v_2 v_3 (-\phi_1 + \phi_2 - \phi_3) = v_1 V_2 V_3 \phi_1 - V_1 v_2 V_3 \phi_2 + V_1 V_2 v_3 \phi_3,\vspace{.2cm}\\
c \, v_1 v_2 v_3 (-\phi_1 - \phi_2 + \phi_3) = v_1 V_2 V_3 \phi_1 + V_1 v_2 V_3 \phi_2 - V_1 V_2 v_3 \phi_3.
\end{array}\right.
\end{equation}
Adding the equations in (\ref{a1a2a3zero}) two by two we get 
$$\frac{V_1}{v_1}=\frac{V_2}{v_2}=\frac{V_3}{v_3} \quad {\text{or}} \quad c=0,$$ what is a contradiction.  \\
\noindent $(b)$ Suppose that  $\alpha_1,\alpha_2$ and $\alpha_3$ do not vanish at some point, and hence on some open subset $U\subset M^3$. Since $\alpha_1,\alpha_2$ and $\alpha_3$ do not vanish at any point of $U$, it follows from (\ref{compat.eq}) that $F_i\equiv 0$ on $U$, $1\leq i\leq 3$, and hence  $\alpha_1^2$, $\alpha_2^2$ and $\alpha_3^2$ satisfy the system of linear equations
\begin{equation}\label{sist.prop1}
\left\{\begin{array}{llll}
F_{11} \alpha_1^2 + F_{12} \alpha_2^2 + F_{13} \alpha_3^2 = - F_{14} - F_{15}, \vspace{0.2cm} \\
F_{21} \alpha_1^2 + F_{22} \alpha_2^2 + F_{23} \alpha_3^2 = - F_{24} - F_{25}, \vspace{0.2cm} \\
F_{31} \alpha_1^2 + F_{32} \alpha_2^2 + F_{33} \alpha_3^2 = - F_{34} - F_{35}.
\end{array}\right.
\end{equation}       
Notice that
\begin{eqnarray*}
\det\left(\begin{array}{lll}
F_{11} & F_{12} & F_{13}  \vspace{0.2cm} \\
F_{21} & F_{22} & F_{23}  \vspace{0.2cm} \\
F_{31} & F_{32} & F_{33} 
\end{array}\right) =0\,\,\, \text{and}\,\,\, 
\det\left(\begin{array}{lll}
F_{14}+F_{15} & F_{12} & F_{13}  \vspace{0.2cm} \\
F_{24}+F_{25} & F_{22} & F_{23}  \vspace{0.2cm} \\
F_{34}+F_{35} & F_{32} & F_{33} 
\end{array}\right)  = a 
\end{eqnarray*}
where
\begin{equation*}
\begin{array}{lll}
a &=& 36 \phi_1^3 \phi_2 \phi_3 v_1^6 V_1^2 \Theta^2(-1 + 3 v_1^2) (1 + 3 v_2^2) (-1 + 3 v_3^2) (v_1^2 + v_2^2) (v_1^2 - v_3^2) 
  (v_2^2 \\  &  & + \, v_3^2)^4 (4 - 15 v_1^2 + 15 v_1^4 + 15 v_2^2 - 30 v_1^2 v_2^2 + 15 v_1^4 v_2^2 + 15 v_2^4 - 15 v_1^2 v_2^4). 
\end{array}
\end{equation*}
Let $P$ be the polynomial of two variables given by
\begin{equation*}
P(v_1,v_2) = 4 - 15 v_1^2 + 15 v_1^4 + 15 v_2^2 - 30 v_1^2 v_2^2 + 15 v_1^4 v_2^2 + 15 v_2^4 - 15 v_1^2 v_2^4.
\end{equation*}
As the system (\ref{sist.prop1}) has a solution, we have that $a=0$ and, therefore,  $P(\psi(p))=0,$ where $\psi:U\to\mathbb{R}^2,$ $\psi(p)=(v_1(p),v_2(p))$, for all $p\in U\subset M^3.$ Since $P^{-1}(0)$ is an algebraic curve, it follows that $\psi$ has rank zero or one at each point. This implies that the determinant of any submatrix of order $2,$ of the matrix
\begin{eqnarray*}
\left(\begin{array}{lll}
\displaystyle\frac{\partial v_i}{\partial u_j}
\end{array}\right)_{1\leq i \leq 2; \, 1\leq j \leq 3} ,
\end{eqnarray*}
is a equal to zero, namely, 
\begin{equation*}
0=v_1^2 \phi_1 - v_2^2 \phi_2 + v_3^2 \phi_3=\Theta,
\end{equation*}
what is a contradiction, by (\ref{good.term}).\\
 
\noindent $(c)$ Let $U \subset M^3$ be an open subset where $\alpha_1$ and $\alpha_2$ vanish identically and $\alpha_3(p)\neq 0$ for all $p\in U$. In this case, the fourth equation in (\ref{red.system}), with $i=1,2,$ and $F_3=0$  reduce, respectively, to
 \begin{equation}\label{al1=al2=0}
\left\{ \begin{array}{l}
A \alpha_3^2 + B + C =0, \\\\

\widetilde{A} \alpha_3^2 + \widetilde{B} + \widetilde{C} =0, \\\\

F_{33} \alpha_3^2 + F_{34}+F_{35}=0.
 \end{array}\right.
 \end{equation}
 where
 \begin{equation*}
 \begin{array}{lll}
\displaystyle  A=\frac{v_1 \phi_1(12 v_1^2 - 15 v_1^4 + 12 v_2^2 + 15 v_2^4 - \phi_3)}{2 \phi_3},  \vspace{.2cm}\\
\displaystyle  B=\frac{v_1 \phi_1 V_2 V_3  - v_2\phi_2 V_1 V_3  - v_3 \phi_3 V_1 V_2 }{2 v_2 v_3 \phi_2 \phi_3} \vspace{.2cm}\\
\displaystyle  C=\frac{c \, v_1 (\phi_1 - \phi_2 - \phi_3)}{2 \phi_2 \phi_3}, \vspace{.2cm}\\
\displaystyle  \widetilde{A}=\frac{v_2  \phi_2 (-12 v_1^2 + 15 v_1^4 - 12 v_2^2 - 15 v_2^4 - \phi_3)}{2 \phi_3}, \vspace{.2cm}\\
\displaystyle  \widetilde{B}=\frac{- v_1\phi_1 V_2 V_3  + v_2 \phi_2 V_1 V_3  - v_3 \phi_3 V_1 V_2 }{2 v_1 v_3 \phi_1 \phi_3} \vspace{.2cm}\\
\displaystyle  \widetilde{C}= \frac{c \, v_2 (-\phi_1 + \phi_2 - \phi_3)}{2 \phi_1 \phi_3}.
 \end{array}
 \end{equation*}

It follows from the equations in  (\ref{al1=al2=0}) that
\begin{equation*}
\det\left(\begin{array}{ccc}
A & B & C \\
\widetilde{A} & \widetilde{B} & \widetilde{C} \\
F_{33} & F_{34} & F_{35}
\end{array}\right)=0.
\end{equation*}

\begin{eqnarray*}
\det\left(\begin{array}{cc}
A & B+C \\
\widetilde{A} & \widetilde{B} + \widetilde{C}
\end{array}\right)=0 \qquad {\text{and}} \qquad \det\left(\begin{array}{cc}
A &  B+C \\
F_{33} & F_{34}+F_{35}
\end{array}\right)=0.
\end{eqnarray*}
Computing the determinants and using the expressions for  $V_2$ and $V_3,$ in terms of $V_1,$ we obtain respectively
\begin{equation}\label{osdet=0}
\left\{\begin{array}{l}  
 2 c v_1 v_2   (1 + 3 v_2^2) ( 1 - v_1^2) (v_1^2 + v_2^2)^2V_1^2 P_1 = 0\\\\
     
\displaystyle\frac{v_2 (v_1^2 + v_2^2)}{ v_1 (v_2^2 +  v_3^2)^2 \phi_3} [V_1^2 P_2 - c v_1^2 (-1 + 9 v_3^2) (v_2^2 + v_3^2)^2]=0\\\\
       
v_1 (v_1^2 +  v_2^2) \phi_1[V_1^2 P_3 +  c v_1^2 (v_2^2 + v_3^2)^2 P_4]=0,
\end{array}\right.
\end{equation}
onde 
\begin{equation*}
\begin{array}{l}  
P_1= -v_2^2 - v_2^4 + 10 v_2^2 v_3^2 + 9 v_2^4 v_3^2 + 8 v_3^4 - 9 v_2^2 v_3^4 - 9 v_3^6 - 27 v_3^8; \vspace{.2cm}\\
     
P_2= v_2^2 + v_2^4 + 3 v_3^2 - v_2^2 v_3^2 - 14 v_3^4 + 18 v_3^6; \vspace{.2cm}\\
       
P_3= -8 v_2^2 - 40 v_2^4 - 85 v_2^6 - 86 v_2^8 - 33 v_2^{10} + 96 v_2^2 v_3^2 + 416 v_2^4 v_3^2  +  763 v_2^6 v_3^2 \vspace{.2cm}\\ 
\hspace{1cm} + \, 672 v_2^8 v_3^2 + 225 v_2^{10} v_3^2 + 16 v_3^4 - 329 v_2^2 v_3^4 - 1128 v_2^4 v_3^4  -  1341 v_2^6 v_3^4 \vspace{.2cm}\\ 
\hspace{1cm} - \, 594 v_2^8 v_3^4 - 50 v_3^6 +  901 v_2^2 v_3^6 + 2226 v_2^4 v_3^6 + 1323 v_2^6 v_3^6 - 2 v_3^8 \vspace{.2cm}\\ 
\hspace{1cm} - \,  2118 v_2^2 v_3^8 - 3960 v_2^4 v_3^8 - 1620 v_2^6 v_3^8 + 90 v_3^{10} +  2268 v_2^2 v_3^{10} \vspace{.2cm}\\ 
\hspace{1cm} + \, 2430 v_2^4 v_3^{10} - 54 v_3^{12} - 810 v_2^2 v_3^{12};\\\\

P_4= 16 v_2^2 + 37 v_2^4 + 21 v_2^6 - 16 v_3^2 -  223 v_2^2 v_3^2 - 441 v_2^4 v_3^2 - 234 v_2^6 v_3^2 \vspace{.2cm}\\ 
\hspace{1cm} + \, 34 v_3^4 +  600 v_2^2 v_3^4 + 1035 v_2^4 v_3^4 + 405 v_2^6 v_3^4 + 36 v_3^6 - 459 v_2^2 v_3^6 \vspace{.2cm}\\ 
\hspace{1cm} - \, 567 v_2^4 v_3^6 - 54 v_3^8 + 162 v_2^2 v_3^8,
\end{array}
\end{equation*}

If $c=0,$  then it follows from the equations in (\ref{osdet=0}) that $P_2=0$ and $P_3=0.$ We have that $P_2^{-1}(0)$ and $P_3^{-1}(0)$ are algebraic curves. For Bezout theorem, $P_2^{-1}(0) \cap P_3^{-1}(0)$ has only a finite number of points. Assuming without loss of generality that $U$ is connected, we conclude that the functions  $v_2$ e $v_3$ are constant in $U$, which is a contradiction, considering item $(a)$ of this proposition.

If $c\neq 0,$ then it follows from the equations (\ref{osdet=0}) that $P_1=0.$ Differentiating the equation $P_1=0$, with respect to $u_3$, we get
\begin{eqnarray}\label{caso.c.3}
\begin{array}{l}
0 = - \, v_2^2 - 16 v_2^4 - 30 v_2^6 - 15 v_2^8 + 14 v_2^2 v_3^2 + 178 v_2^4 v_3^2 + 300 v_2^6 v_3^2 + 135 v_2^8 v_3^2 \vspace{.2cm}\\
\hspace{.8cm} + \, 32 v_3^4 + 92 v_2^2 v_3^4 - 180 v_2^4 v_3^4 - 270 v_2^6 v_3^4 - 134 v_3^6 - 186 v_2^2 v_3^6 + 54 v_2^4 v_3^6 \vspace{.2cm}\\
\hspace{.8cm} - \, 33 v_3^8 - 891 v_2^2 v_3^8 - 972 v_2^4 v_3^8 + 459 v_3^{10} + 972 v_2^2 v_3^{10} - 324 v_3^{12}
\end{array}
\end{eqnarray} 
The same argument, used in the case $c=0$, applied to equations $P_1=0$ and (\ref{caso.c.3}), implies the same conclusion in that case. 

The other possible cases follow in a similar way.

\noindent $(d)$   
 Let $U\subset M^3$ be an open and connected subset where  $\alpha_1$ vanishes identically but  $\alpha_2(p)\neq 0\neq \alpha_3(p)$ for all $p\in U$.  It follows from  (\ref{compat.eq}) and the fourth equation in (\ref{red.system}), when $i=1,$ that $F_2=0$, $F_3=0$ and $\dfrac{\partial \alpha_1}{\partial u_1}=0,$ that is,  $\alpha_2^2$ and $\alpha_3^2$ satisfy the  system of linear equations
\begin{equation}\label{sist.a1=0}
\left\{\begin{array}{rrrrcrl}
A \alpha_2^2 &+& B \alpha_3^2 &+& C &=& 0; \vspace{.2cm}\\
F_{22} \alpha_2^2 &+& F_{23} \alpha_3^2 &+& F_{24}+F_{25} &=& 0; \vspace{.2cm}\\
F_{32} \alpha_2^2 &+& F_{33} \alpha_3^2 &+& F_{34}+F_{35} &=& 0,
\end{array}\right.
\end{equation}
where 
\begin{equation*}
\begin{array}{l}
\displaystyle A = \frac{ v_1 \phi_1 }{2 \phi_2}(12 v_1^2 - 15 v_1^4 - 12 v_3^2 + 15 v_3^4 - \phi_2); \vspace{.2cm}\\
\displaystyle B = \frac{ v_1 \phi_1}{2 \phi_3}(12 v_1^2 - 15 v_1^4 + 12 v_2^2 + 15 v_2^4 - \phi_3); \vspace{.2cm}\\
\displaystyle C = \frac{v_1 V_2 V_3 \phi_1 - V_1 v_2 V_3 \phi_2 - V_1 V_2 v_3 \phi_3}{2 v_2 v_3 \phi_2 \phi_3}.
\end{array}
\end{equation*}
The fact that the system (\ref{sist.a1=0}) has a solution implies that
\begin{equation*}
\det \left(\begin{array}{ccc}
A  & B  & C  \\ 
F_{22}  & F_{23} & F_{24}+F_{25} \\ 
F_{32}  & F_{33}  & F_{34}+F_{35}
\end{array}\right) =0.
\end{equation*}
Computing such determinant we obtain that $P(\psi(p))=0,$ for all $p\in U,$ where $P=P(v_1,v_2)$ is a polynomial in two variables 
\begin{equation}
P=c R_1-\tilde{c} R_2,
\end{equation}
where $R_1$ and $R_2$ are defined in the appendix, and $\psi:U\to\mathbb{R}^2,$ $\psi(p)=(v_1(p),v_2(p))$.
Since $P$ is not the null polynomial, regardless of the values assigned to $c$ and $\tilde{c},$ then $P^{-1}(0)$ is an algebraic curve and $\psi$ has rank zero or one, at each point. This implies that the determinant of any order $2,$ submatrix of the matrix
\begin{eqnarray*}
\left(\begin{array}{lll}
\displaystyle\frac{\partial v_i}{\partial u_j}
\end{array}\right)_{1\leq i \leq 2; \, 1\leq j \leq 3} ,
\end{eqnarray*}
is equal to zero. Any of the determinants implies $\Theta=0,$ contradicting (\ref{good.term}).

\end{proof}


 The following theorem provides the solutions the problem, mentioned in the introduction of this paper, for cases of hypersurfaces $f$ with three nonzero distinct principal curvatures.

\begin{theorem}\label{thm:main}
There exists no minimal hypersurface $f:M^3 \to \Q^4(c)$ with three nonzero distinct principal curvatures and that for which exist another isometric immersion $\tilde{f}:M^3 \to Q^4(\tilde{c})$ with $c \neq \tilde{c}.$
\end{theorem} 
 \begin{proof}
 By Theorem \ref{main.theo} we can assume that  $f$ is a holonomic hypersurface whose associated pair $(v,V)$ satisfies (\ref{main.theo.eq}). Since $f$ is minimal, the pair $(v,V)$ satisfies the four equations in (\ref{caract.minima}). Let $\alpha$ be the function defined in Proposition \ref{prop.key}, which satisfies  $(a)$, $(b)$, $(c)$ and  $(d)$ in Proposition  \ref{propried.of.solu}. By  $(a)$, there exists an open subset  $U\subset M^3$ and  $i\in\{1,2,3\}$ such that $\alpha_i(p)\neq 0$ for all $p\in U$. It follows from  $(c)$  that there exists an open subset  $U'\subset U$ and $j\in\{1,2,3\}$, $j\neq i,$ such that   $\alpha_j(q)\neq 0$ for all $q\in U'.$ Then  $(d)$ implies that there exists  $p_0\in U'$ such that  $\alpha_k(p_0)\neq 0$, $j\neq k\neq i$. Thus  $\alpha_1(p_0)\neq 0$,  $\alpha_2(p_0)\neq 0$ and  $\alpha_3(p_0)\neq 0$, which contradicts $(b)$. 
 \end{proof}

\section{Proof the Theorem \ref{main1}}

Let $f\colon M^{3} \to \Q^{4}(c)$ and  $\tilde{f}\colon M^{3} \to \Q^{4}(\tilde{c})$ be hypersurfaces with  $c \neq \tilde{c}$. The second author and Tojeiro \cite{ct} shown that at each point $x\in M^3$ there exists an  orthonormal basis $\{e_1,e_2,e_{3}\}$ of $T_xM^{3}$ that simultaneously  diagonalizes the second fundamental forms of   $f$ and $\tilde{f}$. Let  $\lambda_1, \lambda_2, \lambda_3$ and $\mu_1, \mu_2, \mu_3$ be the principal curvatures of  $f$ and $\tilde f$ correspondent to  $e_1, e_2$ and $e_3$, respectively. 

Under these conditions, we have the following:

 \begin{lemma}\label{multdois}
 \begin{itemize}
 \item[(a)]  Assume that $f$ is minimal and has a principal curvature of multiplicity two,  say, that $\lambda_1=\lambda_2:=\lambda$. Then 
 $$2\lambda+\lambda_3=0,\,\,\,\, \mu_1=\mu_2:= \mu \neq 0\,\,\,\,  2\mu + \mu_3=\dfrac{3}{\mu}(c-\tilde c).$$
 
 \item[(b)] Assume that $f$ is minimal and has a null principal curvature,  say, that $\lambda_3=0$. Then $c>\tilde c$,  $\mu_1=\mu_2:=\mu$, 
 \be\label{e1}
 c-\tilde{c}-\lambda_1^2=\mu^2
 \ee
 and
 \be\label{e2}
 c-\tilde c= \mu\mu_3.
 \ee
\end{itemize}
 \end{lemma}
 \proof By the Gauss equations for $f$ and $\tilde f$, we have
 \begin{eqnarray}\label{gftildef} 
 c+\epsilon \lambda_i\lambda_j=\tilde c+\tilde \epsilon \mu_i\mu_j, \,\,\,1\leq i\neq j\leq 3.
 \end{eqnarray}
Moreover, as $f$ is minimal, we have 
\be \label{min}
\lambda_1 + \lambda_2+\lambda_3=0
\ee 
 $(a)$ If $\lambda_1=\lambda_2:=\lambda$, then the preceding equations are
 \be\label{c1}
 c+\lambda^2=\tilde c+ \mu_1\mu_2,
 \ee
 \be\label{c2}
 c + \lambda\lambda_3=\tilde c+ \mu_1\mu_3,
 \ee
 \be\label{c3}
 c+ \lambda\lambda_3=\tilde c+ \mu_2\mu_3
 \ee
and
 \be\label{c4}
 \lambda_3=-2\lambda.
 \ee
 The equations (\ref{c2}) and (\ref{c3}) yield $$\mu_3(\mu_1-\mu_2)=0,$$ hence either $\mu_3=0$ or $\mu_1=\mu_2$. In view of (\ref{c2}) and (\ref{c4}), the first possibility can not occur if $c < \tilde{c}$. Thus, in this case we must have $\mu_1=\mu_2:=\mu$, and then $c-\tilde c+ \lambda^2= \mu^2$ by (\ref{c1}) and, substituting (\ref{c4}) into (\ref{c2}), obtain $c-\tilde c - 2\lambda^2=\mu\mu_3$. Isolating $\lambda^2$ in this penultimate equation and substituting in the last one, it follows that $\mu \neq 0$ and $\mu_3=\dfrac{3}{\mu}(c-\tilde c) - 2 \mu.$
  
 Otherwise, either the same conclusion holds or 
 $\mu_3=0$ and, then, $c + \lambda\lambda_3=\tilde{c}$ by (\ref{c2}), and $c-\tilde c = 2\lambda^2$ by (\ref{c2}) and (\ref{c4}). From which it follows that $f$ is isoparametric. By Cartan's formula (see \cite[p.91]{CeRyBook}), we have $c+\lambda\lambda_3=0$, which leads us to a contradiction.\\
 
$(b)$ If $\lambda_3=0$, then is equations (\ref{gftildef}) and (\ref{min}) become
\be\label{d1}
 c-\tilde c -  \lambda_1^2= \mu_1\mu_2,
 \ee
 \be\label{d2}
 c-\tilde c= \mu_1\mu_3
 \ee
 and 
 \be\label{d3}
c-\tilde c= \mu_2\mu_3.
 \ee
 Since $\mu_3\neq 0$ by (\ref{d2}) or (\ref{d3}), these equations imply that  $\mu_1=\mu_2:=\mu$, and we obtain
  (\ref{e2}). From (\ref{d1}), we obtain $c-\tilde{c}=\lambda_1^2+\mu^2$ from which it follows that $c > \tilde{c}$. \qed

The theorem below gives us a solution of problem mentioned in the introduction of this paper, in cases where the hypersurface $f$ has a principal curvature with multiplicities or zero.

\begin{theorem}\label{thm:dual}
Let $f:M^3 \to \Q^4(c)$ be a minimal hypersurface for which there exists an isometric immersion $\tilde{f}:M^3 \to \Q^4(\tilde{c})$ with $c \neq \tilde{c}$.

\begin{enumerate}
    \item [a)] If $f$ has a principal curvature of multiplicity three, then $c>\tilde c$, $f$ is a totally geodesic and $\tilde f$ is a umbilic non totally geodesic.
    \item [b)] Assume that $f$ has a principal curvature of multiplicity two, then $f$ is a rotation hypersurface whose profile curve is a $c\delta-$catenary in a totally geodesic surface $\Q^2(c)$ of $\Q^4(c)$ and $\tilde{f}$ is a rotation hypersurface whose profile curve is a $c\tilde{\delta}-$catenary in a totally geodesic surface $\Q^2(\tilde{c})$ of $\Q^4(\tilde{c})$.

    \item [c)] If one of the principal curvatures of $f$ is zero, then $c > \tilde{c}$, $f$ is a generalized cone over a minimal surface in an umbilical hypersurface $\Q^3(\bar{c})$ of $\Q^4(c)$, $\bar{c} \geq c$, and $\tilde{f}$ is a rotation hypersurface whose profile curve is a $c-$helix in a totally geodesic surface $\Q^2(\tilde{c})$ of $\Q^4(\tilde{c})$.
\end{enumerate}
\end{theorem} 
\begin{proof} The case $a)$ is trivial. Assume that  $f$ has a principal curvature of multiplicity two, say, $\lambda_1=\lambda_2:=\lambda$. Then,  it follows from Lemma \ref{multdois} that
\begin{eqnarray}\label{mult2dem}
    \lambda_3=-2\lambda,\,\,\,\, \mu_1=\mu_2:= \mu \neq 0\,\,\,\,  \mu_3=\dfrac{3}{\mu}(c-\tilde c) - 2 \mu.
 \end{eqnarray}   
 Where follows from Theorem 4.2 in \cite{dcd} that $f$ and $\tilde{f}$ are rotation hypersurfaces. Moreover, using the principal curvatures given in Proposition 3.2 in \cite{dcd}, it is shown that the first and third equations in (\ref{mult2dem}) are satisfied if and only if the profiles curves the $f$ and $\tilde f$ are $c\delta-$catenary and $c\tilde{\delta}-$catenary, respectively, in $\Q^2(c)$ and $\Q^2(\tilde{c}).$ 

Finally, if one of the principal curvatures of $f$ is zero then,  it follows from Lemma \ref{multdois} that
$c>\tilde{c}$, $\mu_1=\mu_2:= \mu$, $c-\tilde{c}-\lambda_1^2=\mu^2$ and $c-\tilde{c}=\mu\mu_3.$ Now, proceeding as in the proof of Theorem 4 in \cite{ct} that $f(M^3)$ is contained in a generalized cone over a surface with constant curvature in an umbilical hypersurface $\Q^3(\bar{c})$ of $\Q^4(c)$, with $\bar{c} > c$. On the other hand, follow the Theorem 4.2 in \cite{dcd}  that $\tilde{f}$ is a rotation hypersurface. Moreover, using the principal curvatures given in Proposition 3.2 in \cite{dcd}, it is shown that the equation $c-\tilde{c}=\mu\mu_3$ is satisfied if and only if the profile curve is a $c-$ helix in $\Q^2(\tilde{c}).$
\end{proof}

We are now in a position to prove the main theorem.\\

\noindent \emph{Proof of Theorem \ref{main1}}:  
Follow the Theorems \ref{thm:main} and \ref{thm:dual}.  \qed

\section{Appendix A}

\begin{equation*}
\begin{array}{rlll}
F_1 \!\! &=& \!\!\! F_{11} \alpha_1^2 + F_{12} \alpha_2^2 + F_{13} \alpha_3^2 + F_{14} + F_{15}; \hspace{8cm} \vspace{.2cm}\\
F_2 \!\! &=& \!\!\! F_{21} \alpha_1^2 + F_{22} \alpha_2^2 + F_{23} \alpha_3^2 + F_{24} + F_{25}; \vspace{.2cm}\\
F_3 \!\! &=& \!\!\! F_{31} \alpha_1^2 + F_{32} \alpha_2^2 + F_{33} \alpha_3^2 + F_{34} + F_{35};
\end{array}
\end{equation*}

\begin{equation*}
\begin{array}{rll}
F_{11} \!\!\! &=& \!\!\! - \ 3 v_1^2  \phi_2 \phi_3 (v_2^2 + v_3^2) (v_2^2 + v_3^2) (2 v_1^2 - 7 v_1^4 + 8 v_1^6 - 3 v_1^8 - 2 v_2^2 + 17 v_1^2 v_2^2 \hspace{3cm}\\
       && \!\!\! - \ 15 v_1^4 v_2^2 - 45 v_1^6 v_2^2 + 45 v_1^8 v_2^2 - 2 v_2^4 +  15 v_1^2 v_2^4 - 45 v_1^6 v_2^4);\\
     \end{array}
\end{equation*}
\begin{equation*}
\begin{array}{rll}
F_{12} \!\!\! &=& \!\!\! - \ 3 v_1^2 \phi_1^2 \phi_3 (v_2^2 + v_3^2) (2 v_1^2 - 8 v_1^4 + 6 v_1^6 - 2 v_2^2 + 21 v_1^2 v_2^2 - 53 v_1^4 v_2^2  \hspace{3cm}\\
       && \!\!\! + \ 30 v_1^6 v_2^2 -  5 v_2^4 + 34 v_1^2 v_2^4 - 45 v_1^4 v_2^4 - 3 v_2^6 + 15 v_1^2 v_2^6);\\
     \end{array}
\end{equation*}
\begin{equation*}
\begin{array}{rll}
F_{13} \!\!\! &=& \!\!\! 3 v_1^2 \phi_1^2 \phi_2 (v_2^2 + v_3^2) (-v_1^2 + 2 v_1^4 - v_1^6 +  v_2^2 - 6 v_1^2 v_2^2 - 6 v_1^4 v_2^2 + 15 v_1^6 v_2^2  \hspace{3cm}\\
       && \!\!\! + \ 4 v_2^4 - 20 v_1^2 v_2^4 + 3 v_2^6 - 15 v_1^2 v_2^6);\\
\end{array}
\end{equation*}
\begin{equation*}
\begin{array}{rll}
F_{14} \!\!\! &=& \!\!\! V_1^2 (-5 v_1^4 + 13 v_1^6 + 7 v_1^8 - 33 v_1^{10} + 18 v_1^{12} + 4 v_2^2 - 48 v_1^2 v_2^2 + 178 v_1^4 v_2^2 \hspace{3cm}\\
       && \!\!\! - \ 428 v_1^6 v_2^2   + 780 v_1^8 v_2^2 - 756 v_1^{10} v_2^2 + 270 v_1^{12} v_2^2 + 17 v_2^4 - 169 v_1^2 v_2^4 \\
       && \!\!\! + \ 408 v_1^4 v_2^4 - 384 v_1^6 v_2^4 + 279 v_1^8 v_2^4 - 135 v_1^{10} v_2^4 + 26 v_2^6 - 224 v_1^2 v_2^6 \\
       && \!\!\! + \ 324 v_1^4 v_2^6 + 144 v_1^6 v_2^6 - 270 v_1^8 v_2^6 + 13 v_2^8 - 99 v_1^2 v_2^8 + 63 v_1^4 v_2^8 \\
       && \!\!\! + \ 135 v_1^6 v_2^8);\\
\end{array}
\end{equation*}
\begin{equation*}
\begin{array}{rll}
F_{15} \!\!\! &=& \!\!\! -c \, v_1^2 (v_2^2+v_3^2)^2 (-5 v_1^2 + 8 v_1^4 + 15 v_1^6 - 18 v_1^8 + 5 v_2^2 - 77 v_1^2 v_2^2 \hspace{3cm}\\
    && \!\!\! + \ 225 v_1^4 v_2^2 - 207 v_1^6 v_2^2 + 
    54 v_1^8 v_2^2 + 5 v_2^4 - 72 v_1^2 v_2^4 + 153 v_1^4 v_2^4 - 
    54 v_1^6 v_2^4)\\
\end{array}
\end{equation*}
\begin{equation*}
\begin{array}{rll}     
F_{21} \!\!\! &=& \!\!\! - \ 3 v_1^2  \phi_2^2 \phi_3 (v_2^2 + v_3^2)^2 (v_2^2 + 2 v_2^4 + v_2^6 - v_3^2 - 6 v_2^2 v_3^2 + 6 v_2^4 v_3^2 + 15 v_2^6 v_3^2 \hspace{3cm} \\
       && \!\!\! + \ 4 v_3^4 + 20 v_2^2 v_3^4 - 3 v_3^6 - 15 v_2^2 v_3^6);\\
     \end{array}
\end{equation*}
\begin{equation*}
\begin{array}{rll} 
F_{22} \!\!\! &=& \!\!\! 3 v_1^2 \phi_1 \phi_3 (v_1^2 - v_3^2) (v_2^2 + v_3^2)^2 (-2 v_2^2 - 7 v_2^4 - 8 v_2^6 - 3 v_2^8 + 2 v_3^2 + 17 v_2^2 v_3^2 \hspace{3cm}\\
       && \!\!\! + \ 15 v_2^4 v_3^2 - 45 v_2^6 v_3^2 - 45 v_2^8 v_3^2 - 2 v_3^4 - 15 v_2^2 v_3^4 + 45 v_2^6 v_3^4);\\
     \end{array}
\end{equation*}
\begin{equation*}
\begin{array}{lll} 
 F_{23} \!\!\! &=& \!\!\!  3  v_1^2 \phi_1 \phi_2^2(v_2^2 + v_3^2)^2 (-2 v_2^2 - 8 v_2^4 - 6 v_2^6 + 2 v_3^2 + 21 v_2^2 v_3^2 + 53 v_2^4 v_3^2 \hspace{3cm} \\
       && \!\!\! + \ 30 v_2^6 v_3^2 - 5 v_3^4 - 34 v_2^2 v_3^4 - 45 v_2^4 v_3^4 + 3 v_3^6 + 15 v_2^2 v_3^6);\\
     \end{array}
\end{equation*}
\begin{equation*}
\begin{array}{rll} 
F_{24} \!\!\! &=& \!\!\! V_1^2 (v_1^2 - v_3^2)(5 v_2^4 + 13 v_2^6 - 7 v_2^8 - 33 v_2^{10} - 18 v_2^{12} + 4 v_3^2 + 48 v_2^2 v_3^2 \hspace{3cm} \\
       && \!\!\! + \ 178 v_2^4 v_3^2 + 428 v_2^6 v_3^2 + 780 v_2^8 v_3^2 + 756 v_2^{10} v_3^2 + 270 v_2^{12} v_3^2 - 17 v_3^4 \\
       && \!\!\! - \ 169 v_2^2 v_3^4 - 408 v_2^4 v_3^4 - 384 v_2^6 v_3^4 - 279 v_2^8 v_3^4 - 135 v_2^{10} v_3^4 + 26 v_3^6 \\
       && \!\!\! + \ 224 v_2^2 v_3^6 + 324 v_2^4 v_3^6 - 144 v_2^6 v_3^6 - 270 v_2^8 v_3^6 - 13 v_3^8 - 99 v_2^2 v_3^8 \\
       && \!\!\! - \ 63 v_2^4 v_3^8 + 135 v_2^6 v_3^8);\\
\end{array}
\end{equation*} 
\begin{equation*}
\begin{array}{rll}
F_{25} \!\!\! &=& \!\!\!  -c \, v_1^2 (v_1^2 - v_3^2) (v_2^2 + v_3^2)^2 (-5 v_2^2 - 8 v_2^4 + 15 v_2^6 + 18 v_2^8 + 5 v_3^2 + 77 v_2^2 v_3^2 \hspace{3cm}\\
& & \!\!\! + \ 225 v_2^4 v_3^2 + 207 v_2^6 v_3^2 + 54 v_2^8 v_3^2 - 5 v_3^4 - 72 v_2^2 v_3^4 - 153 v_2^4 v_3^4 - 54 v_2^6 v_3^4)  \\
\end{array}
\end{equation*}    
\begin{equation*}
\begin{array}{rll}     
F_{31} \!\!\! &=& \!\!\! 3 v_1^2  \phi_2 \phi_3^2 (v_2^2 + v_3^2)^2 (-2 v_1^2 + 5 v_1^4 - 3 v_1^6 - 2 v_3^2 + 21 v_1^2 v_3^2 - 34 v_1^4 v_3^2 \hspace{3cm}\\
       && \!\!\! + \ 15 v_1^6 v_3^2 + 8 v_3^4 - 53 v_1^2 v_3^4 + 45 v_1^4 v_3^4 - 6 v_3^6 + 30 v_1^2 v_3^6);\\
     \end{array}
\end{equation*}
\begin{equation*}
\begin{array}{rll}
F_{32} \!\!\! &=& \!\!\! 3 v_1^2 \phi_1 \phi_3^2 (v_2^2 + v_3^2)^2 (-v_1^2 + 4 v_1^4 - 3 v_1^6 - v_3^2 + 6 v_1^2 v_3^2 - 20 v_1^4 v_3^2 + 15 v_1^6 v_3^2 \hspace{3cm} \vspace{.2cm}\\ 
       && \!\!\! + \ 2 v_3^4 + 6 v_1^2 v_3^4 - v_3^6 - 15 v_1^2 v_3^6) ;\\
     \end{array}
\end{equation*}
\begin{equation*}
\begin{array}{rll}
F_{33} \!\!\! &=& \!\!\! - \ 3  v_1^2 \phi_1 \phi_2 (v_1^2 + v_2^2) (v_2^2 + v_3^2)^2 (-2 v_1^2 + 2 v_1^4 - 2 v_3^2 + 17 v_1^2 v_3^2 - 15 v_1^4 v_3^2 \hspace{3cm}\\
       && \!\!\! + \ 7 v_3^4 - 15 v_1^2 v_3^4 - 8 v_3^6 - 45 v_1^2 v_3^6 + 45 v_1^4 v_3^6 + 3 v_3^8 + 45 v_1^2 v_3^8);\\
     \end{array}
\end{equation*}
\begin{equation*}
\begin{array}{rll}
F_{34} \!\!\! &=& \!\!\! - \ V_1^2 (v_1^2 + v_2^2) (-4 v_1^2 + 17 v_1^4 - 26 v_1^6 + 13 v_1^8 + 48 v_1^2 v_3^2 - 169 v_1^4 v_3^2 \hspace{3cm}\\
       && \!\!\! + \ 224 v_1^6 v_3^2 - 99 v_1^8 v_3^2 - 5 v_3^4 - 178 v_1^2 v_3^4 + 408 v_1^4 v_3^4 - 324 v_1^6 v_3^4 + 63 v_1^8 v_3^4 \\
       && \!\!\! + \ 13 v_3^6 + 428 v_1^2 v_3^6 - 384 v_1^4 v_3^6 - 144 v_1^6 v_3^6 + 135 v_1^8 v_3^6 + 7 v_3^8 - 780 v_1^2 v_3^8 \\
       && \!\!\! + \ 279 v_1^4 v_3^8 + 270 v_1^6 v_3^8 - 33 v_3^{10} + 756 v_1^2 v_3^{10} - 135 v_1^4 v_3^{10} + 18 v_3^{12} \\
       && \!\!\! - \ 270 v_1^2 v_3^{12}) ;
\end{array}
\end{equation*}
\begin{equation*}
\begin{array}{rll}
F_{35} \!\!\! &=& \!\!\! -c \, v_1^2 (v_1^2 + v_2^2) (v_2^2 + v_3^2)^2 (5 v_1^2 - 5 v_1^4 + 5 v_3^2 - 77 v_1^2 v_3^2 + 72 v_1^4 v_3^2 - 8 v_3^4 \hspace{3cm}\\
&& \!\!\! + \, 225 v_1^2 v_3^4 - 153 v_1^4 v_3^4 - 15 v_3^6 - 207 v_1^2 v_3^6 + 54 v_1^4 v_3^6 + 18 v_3^8 + 54 v_1^2 v_3^8)   \\
\end{array}
\end{equation*}

{ \fontsize{9.2pt}{\baselineskip}\selectfont
\begin{equation*}
\begin{array}{lll}
R_1 =  v_1^2 (-1 + 3 v_1^2) (1 + 3 v_2^2) (1 - 3 v_3^2) (57 v_1^6 -  474 v_1^8 + 1620 v_1^{10} - 2910 v_1^{12} + 2895 v_1^{14} - 1512 v_1^{16} \vspace{.2cm}\\ 
\hspace{.8cm} + \, 324 v_1^{18} - v_1^4 v_2^2 + 1173 v_1^6 v_2^2 - 9096 v_1^8 v_2^2 + 28966 v_1^{10} v_2^2 - 48345 v_1^{12} v_2^2 + 44529 v_1^{14} v_2^2\vspace{.2cm}\\ 
\hspace{.8cm} - \, 21438 v_1^{16} v_2^2 + 4212 v_1^{18} v_2^2 + 15 v_1^2 v_2^4 - 111 v_1^4 v_2^4 + 10017 v_1^6 v_2^4 - 70074 v_1^8 v_2^4 + 203709 v_1^{10} v_2^4 \vspace{.2cm}\\ 
\hspace{.8cm} - \, 308985 v_1^{12} v_2^4 + 256473 v_1^{14} v_2^4 -   109998 v_1^{16} v_2^4 + 18954 v_1^{18} v_2^4 - 7 v_2^6 + 231 v1^2 v_2^6 - 1135 v_1^4 v_2^6 \vspace{.2cm}\\ 
\hspace{.8cm} + \, 45029 v_1^6 v_2^6 - 280311 v_1^8 v_2^6 +  728529 v_1^{10} v_2^6 - 975825 v_1^{12} v_2^6 + 703323 v_1^{14} v_2^6 -  256770 v_1^{16} v_2^6 \vspace{.2cm}\\ 
\hspace{.8cm} + \, 36936 v_1^{18} v_2^6 - 51 v_2^8 + 1167 v_1^2 v_2^8 -  4313 v_1^4 v_2^8 + 117744 v_1^6 v_2^8 - 646329 v_1^8 v_2^8 +  1455723 v_1^{10} v_2^8 \vspace{.2cm}\\ 
\hspace{.8cm} - \, 1637577 v_1^{12} v_2^8 + 945432 v_1^{14} v_2^8 -  254178 v_1^{16} v_2^8 + 21870 v_1^{18} v_2^8 - 141 v_2^{10} + 2721 v_1^2 v_2^{10} \vspace{.2cm}\\ 
\hspace{.8cm} - \, 7452 v_1^4 v_2^{10} + 187842 v_1^6 v_2^{10} -  897957 v_1^8 v_2^{10} + 1677429 v_1^{10} v_2^{10} - 1471770 v_1^{12} v_2^{10} \vspace{.2cm}\\ 
\hspace{.8cm} + \, 596808 v_1^{14} v_2^{10} - 87480 v_1^{16} v_2^{10} - 187 v_2^{12} + 3240 v_1^2 v_2^{12} - 5964 v_1^4 v_2^{12} + 184050 v_1^6 v_2^{12} \vspace{.2cm}\\ 
\hspace{.8cm} - \, 741069 v_1^8 v_2^{12} + 1065474 v_1^{10} v_2^{12} - 634716 v_1^{12} v_2^{12} + 131220 v_1^{14} v_2^{12} - 120 v_2^{14} + 1920 v_1^2 v_2^{14} \vspace{.2cm}\\ 
\hspace{.8cm} - \, 1800 v_1^4 v_2^{14} + 103680 v_1^6 v_2^{14} - 327240 v_1^8 v_2^{14} +  311040 v_1^{10} v_2^{14} - 87480 v_1^{12} v_2^{14} - 30 v_2^{16} \vspace{.2cm}\\ 
\hspace{.8cm} + \,  450 v_1^2 v_2^{16} + 25920 v_1^6 v_2^{16} - 55890 v_1^8 v_2^{16} +  21870 v_1^{10} v_2^{16})
\end{array}
\end{equation*}
\begin{equation*}
\begin{array}{l}
R_2 = 106 v_1^8 - 1127 v_1^{10} + 5395 v_1^{12} - 15190 v_1^{14} + 27362 v_1^{16} - 31999 v_1^{18} + 23499 v_1^{20} - 9828 v_1^{22} \vspace{.2cm}\\ 
\hspace{.8cm} + \, 1782 v_1^{24} + 14 v_1^6 v_2^2 + 2074 v_1^8 v_2^2 - 21134 v_1^{10} v_2^2 +  94348 v_1^{12} v_2^2 - 246178 v_1^{14} v_2^2 + 410456 v_1^{16} v_2^2 \vspace{.2cm}\\ 
\hspace{.8cm} - \, 444186 v_1^{18} v_2^2 + 301428 v_1^{20} v_2^2 - 116100 v_1^{22} v_2^2 + 19278 v_1^{24} v_2^2 + 30 v_1^4 v_2^4 - 50 v_1^6 v_2^4 + 18800 v_1^8 v_2^4 \vspace{.2cm}\\ 
\hspace{.8cm} - \, 168840 v_1^{10} v_2^4 + 669656 v_1^{12} v_2^4 - 1544558 v_1^{14} v_2^4 + 2264802 v_1^{16} v_2^4 - 2140020 v_1^{18} v_2^4 \vspace{.2cm}\\ 
\hspace{.8cm} + \, 1250910 v_1^{20} v_2^4 - 404676 v_1^{22} v_2^4 + 53946 v_1^{24} v_2^4 - 30 v_1^2 v_2^6 + 1136 v_1^4 v_2^6 - 9324 v_1^6 v_2^6 \vspace{.2cm}\\ 
\hspace{.8cm} + \, 144902 v_1^8 v_2^6 - 928388 v_1^{10} v_2^6 + 2980324 v_1^{12} v_2^6 - 5634936 v_1^{14} v_2^6 + 6707592 v_1^{16} v_2^6 \vspace{.2cm}\\ 
\hspace{.8cm} - \, 5034366 v_1^{18} v_2^6 + 2239164 v_1^{20} v_2^6 - 502524 v_1^{22} v_2^6 + 36450 v_1^{24} v_2^6 +  8 v_2^8 - 559 v_1^2 v_2^8 \vspace{.2cm}\\ 
\hspace{.8cm} + \, 13031 v_1^4 v_2^8 - 108876 v_1^6 v_2^8 + 871930 v_1^8 v_2^8 - 3971216 v_1^{10} v_2^8 + 9961902 v_1^{12} v_2^8 - 14762808 v_1^{14} v_2^8 \vspace{.2cm}\\ 
\hspace{.8cm} + \, 13332978 v_1^{16} v_2^8 - 7134885 v_1^{18} v_2^8 + 2001591 v_1^{20} v_2^8 - 204120 v_1^{22} v_2^8 + 62 v_2^{10} -  3254 v_1^2 v_2^{10} \vspace{.2cm}\\ 
\hspace{.8cm} + \, 63546 v_1^4 v_2^{10} - 505742 v_1^6 v_2^{10} + 3147158 v_1^8 v_2^{10} - 11432058 v_1^{10} v_2^{10} +  23188734 v_1^{12} v_2^{10} \vspace{.2cm}\\ 
\hspace{.8cm} - \, 27073386 v_1^{14} v_2^{10} + 18102852 v_1^{16} v_2^{10} - 6362712 v_1^{18} v_2^{10} + 874800 v_1^{20} v_2^{10} +  198 v_2^{12} \vspace{.2cm}\\ 
\hspace{.8cm} - \, 9198 v_1^2 v_2^{12} + 162680 v_1^4 v_2^{12} - 1217554 v_1^6 v_2^{12} + 6586536 v_1^8 v_2^{12} - 20185578 v_1^{10} v_2^{12} \vspace{.2cm}\\ 
\hspace{.8cm} + \, 33479568 v_1^{12} v_2^{12} - 30142206 v_1^{14} v_2^{12} + 13875786 v_1^{16} v_2^{12} - 2536920 v_1^{18} v_2^{12} + 332 v_2^{14} \vspace{.2cm}\\ 
\hspace{.8cm} - \, 14380 v_1^2 v_2^{14} + 237416 v_1^4 v_2^{14} - 1658148 v_1^6 v_2^{14} + 8098920 v_1^8 v_2^{14} - 21099636 v_1^{10} v_2^{14} \vspace{.2cm}\\ 
\hspace{.8cm} + \, 27662472 v_1^{12} v_2^{14} - 17469756 v_1^{14} v_2^{14} + 4242780 v_1^{16} v_2^{14} + 308 v_2^{16} - 12737 v_1^2 v_2^{16} \vspace{.2cm}\\ 
\hspace{.8cm} + \, 199167 v_1^4 v_2^{16} - 1286820 v_1^6 v_2^{16} + 5770710 v_1^8 v_2^{16} - 12460149 v_1^{10} v_2^{16} + 11744919 v_1^{12} v_2^{16} \vspace{.2cm}\\ 
\hspace{.8cm} - \, 3958470 v_1^{14} v_2^{16} + 150 v_2^{18} - 6000 v_1^2 v_2^{18} + 89550 v_1^4 v_2^{18} - 529200 v_1^6 v_2^{18} + 2207250 v_1^8 v_2^{18} \vspace{.2cm}\\ 
\hspace{.8cm} - \, 3693600 v_1^{10} v_2^{18} + 1931850 v_1^{12} v_2^{18} + 30 v_2^{20} - 1170 v_1^2 v_2^{20} + 16740 v_1^4 v_2^{20} - 89100 v_1^6 v_2^{20}\vspace{.2cm}\\ 
\hspace{.8cm}  + \, 352350 v_1^8 v_2^{20} - 386370 v_1^{10} v_2^{20}
\end{array}
\end{equation*}}

%

{\renewcommand{\baselinestretch}{1}
\vspace*{1cm}\begin{tabular}{ll}
Universidade Federal de Alagoas & Universidade Federal de Sergipe \\
Instituto de Matem\'atica & Departamento de Matem\'atica\\
57072-900 -- Macei\'o -- AL -- Brazil & 49100-000 -- S\~ao Crist\'ov\~ao -- SE -- Brazil \\
E-mail: carlos.filho@im.ufal.br & E-mail: samuel@mat.ufs.br
\end{tabular} }


\end{document}